\def\smart{0} % 1 for smartphones, 0 for A4 paper

\documentclass[12pt,a4paper,oneside]{amsart}

\usepackage{ifthen}

\usepackage{color,amssymb,latexsym,amsfonts,textcomp,fancyhdr,calc,graphicx}
\usepackage{amsmath,amstext,amsthm,amssymb,amsxtra}

\ifthenelse{\smart=0}{
\setlength{\textwidth}{\paperwidth}
\addtolength{\textwidth}{-2in}
\calclayout
}{
\usepackage[paperwidth=10cm,paperheight=20cm,left=5mm,right=5mm,top=10mm,bottom=5mm]{geometry}
}

\usepackage[colorlinks,citecolor=red,pagebackref,hypertexnames=false]{hyperref}
\usepackage{dsfont}
\usepackage[framemethod=tikz]{mdframed}
\usepackage[inline]{asymptote}

\usepackage{amsfonts}
\usepackage{amsmath}
\usepackage{amssymb}
\usepackage{graphicx}
\usepackage{amscd}
\usepackage{xcolor}
\usepackage{hyperref}
\usepackage[normalem]{ulem}
\usepackage{bbm}
\usepackage{enumitem}

\usepackage{txfonts,pxfonts,tikz} %also txfonts 
\usepackage{tgschola}
\usepackage[T1]{fontenc}

\usepackage{color} \usepackage{soul}

\numberwithin{equation}{section}

\theoremstyle{plain}
\newtheorem{theorem}{Theorem}[section]

\theoremstyle{definition}
\newtheorem{definition}{Definition}[section]
\newtheorem{remark}{Remark}[section]

\newtheorem{case[theorem]}{Case}

\numberwithin{equation}{section}

\newcommand{\beql}[1]{\begin{equation}\label{#1}}
\newcommand{\eeq}{\end{equation}}
\newcommand{\comment}[1]{}

\newcommand{\Abs}[1]{{\left|{#1}\right|}}

\newcommand{\Mean}{{\mathbb{E}}}
\newcommand{\Floor}[1]{{\left\lfloor{#1}\right\rfloor}}

\newcommand{\Prob}[1]{{{\mathbb{P}}\left[{#1}\right]}}
\newcommand{\Set}[1]{{\left\{{#1}\right\}}}

\newcommand{\RR}{{\mathbb R}}

\newcommand{\NN}{{\mathbb N}}

\newcommand{\diam}{{\rm diam\,}}

\makeatletter
\@namedef{subjclassname@2020}{\textup{2020} Mathematics Subject Classification}
\makeatother

\begin{document}
	
\title{Sets of full measure avoiding Cantor sets}

\author{Mihail N. Kolountzakis}
\address{Department of Mathematics and Applied Mathematics, University of Crete, Voutes Campus, 70013 Heraklion, Crete, Greece.}
\email{kolount@uoc.gr}

\date{September 13, 2022}

\thanks{Supported by the Hellenic Foundation for Research and Innovation, Project HFRI-FM17-1733 and by University of Crete Grant 4725.
Part of this work was done while the author visited the Ghent Analysis and PDE Center, at the University of Ghent, whose support is gratefully acknowledged.}

\begin{abstract}
In relation to the Erd\H os similarity problem (show that for any infinite set $A$ of real numbers there exists a set of positive Lebesgue measure which contains no affine copy of $A$) we give some new examples of infinite sets which are not universal in measure, i.e. they satisfy the above conjecture. These are symmetric Cantor sets $C$ which can be quite thin: the length of the $n$-th generation intervals defining the Cantor set is decreasing almost doubly exponentially. Further, we achieve to construct a set, not just of positive measure, but of \textit{full measure} not containing any affine copy of $C$. Our method is probabilistic.
\end{abstract}



\maketitle

\tableofcontents

\setlength{\parskip}{0.5em}

\sloppy

%%%%%%%%%%%%%%%%%%%%%%%%%%%%%%%%%%%%%%%%%%%%%%%%%%%%%%%%%%%%%%%%%%%%%
%%%%%%%%%%%%%%%%%%%%%%%%%%%%%%%%%%%%%%%%%%%%%%%%%%%%%%%%%%%%%%%%%%%%%
%%%%%%%%%%%%%%%%%%%%%%%%%%%%%%%%%%%%%%%%%%%%%%%%%%%%%%%%%%%%%%%%%%%%%
%%%%%%%%%%%%%%%%%%%%%%%%%%%%%%%%%%%%%%%%%%%%%%%%%%%%%%%%%%%%%%%%%%%%%

\section{Introduction}

%%%%%%%%%%%%%%%%%%%%%%%%%%%%%%%%%%%%%%%%%%%%%%%%%%%%%%%%%%%%%%%%%%%%%
%%%%%%%%%%%%%%%%%%%%%%%%%%%%%%%%%%%%%%%%%%%%%%%%%%%%%%%%%%%%%%%%%%%%%
\subsection{The Erd\H os similarity problem}
We are interested in is the so-called Erd\H os similarity problem:  given a set $A \subseteq \RR$ when can we find a Lebesgue measurable set $E$ \textit{of positive measure} which contains no affine copy, $x+tA$, of the set $A$ (where $x, t \in \RR, t \neq 0$).

\begin{definition}
Let us call a set $A \subseteq \RR$ \textit{universal in measure} if every Lebesgue measurable set $E \subseteq \RR$ of positive measure contains an affine copy of $A$.
\end{definition}

Obviously any unbounded set $A$ can be avoided by the interval $[0, 1]$, so we assume from now on that $A$ is bounded and, to simplify matters more, that $A \subseteq [0, 1]$.

It is easy to see, looking near a point of density of a set $E$ of positive measure, that every finite set $A$ is universal in measure. Erd\H os has conjectured that no infinite set is universal in measure. This is still open apart from special cases \cite{eigen1985putting,falconer1984problem,kolountzakis1997infinite,bourgain1987construction,gallagher2022topological,humke1998visit,komjath1983large}. See also the survey \cite{svetic2000erdHos} and the related papers \cite{cruz2022large,denson2021large,fraser2018large,maga2011full,mathe2017sets,shmerkin2017salem,yavicoli2021large,bradford2022large,kolountzakis2022large,burgin2022large}.

To prove that there are no infinite universal sets it would suffice to prove that no countable set of the form
$$
A = \Set{a_1 > a_2 > a_3 > \cdots}, \text{ with } a_n \to 0,
$$
is universal. This is known under several conditions on $a_n$ which prevent $a_n$ from converging too rapidly. For instance \cite{eigen1985putting,falconer1984problem} it is known that $A$ is not universal if
\beql{slow}
\frac{a_{n+1}}{a_n} \to 1.
\eeq
In contrast, it is still unknown if the sequence $2^{-n}$ or any other exponentially decreasing sequence is universal. In almost all the existing work on this problem the rapid decay of the sequence $a_n$ presents a problem\footnote{The only exception we know is Theorem 2 in \cite{kolountzakis1997infinite}, where one passes to a fast decaying subsequence in order to exploit the pseudo-random properties of its dilates modulo a fixed length and thus construct a set $E$ of positive measure which does not contain $x+tA$ for \textit{almost all $t \in \RR$}.} and exponential decay is the borderline case that nobody seems to know how to handle.

In \cite{kolountzakis1997infinite} the following result was proved, which easily implies non-universality under condition \eqref{slow}, but is somewhat more flexible especially when $A$ does not have the structure of a convergent sequence.
\begin{theorem}[\cite{kolountzakis1997infinite}]\label{th:old}
Let $A\subseteq\RR$ be an infinite set which contains, for
arbitrarily large $n$, a subset $\{a_1,\ldots,a_n\}$ with
$a_1 > \cdots > a_n > 0$ and
\begin{equation}\label{eq:delta-assumption}
-\log\delta_n = o(n),
\end{equation}
where\footnote{In \cite{kolountzakis1997infinite} the denominator in \eqref{delta} is $a_1$, not $a_1-a_n$, but since we can translate $A$ without changing the problem, these are equivalent formulations. See also \cite{chlebik2015erdos}.}
\beql{delta}
\delta_n = \min_{i=1,\ldots,n-1} {a_i - a_{i+1} \over a_1-a_n}.
\eeq
Then $A$ is not universal in measure.
\end{theorem}
We can see now that any set $A$ of positive Lebesgue measure is not universal. Indeed for any $n$ we can find an affine image of $\Set{1, 2, \ldots, n}$ in $A$, since every finite set is universal, as explained above. Now we apply Theorem \ref{th:old} for this set to obtain that $A$ is not universal.

%%%%%%%%%%%%%%%%%%%%%%%%%%%%%%%%%%%%%%%%%%%%%%%%%%%%%%%%%%%%%%%%%%%%%
%%%%%%%%%%%%%%%%%%%%%%%%%%%%%%%%%%%%%%%%%%%%%%%%%%%%%%%%%%%%%%%%%%%%%
\subsection{Uncountable and Cantor sets}

It makes sense to ask the Erd\H os similarity problem under the additional assumption that the set $A$ (all of whose affine copies are to be avoided by a set of positive measure) is not just infinite but even \textit{uncountable}. No one really knows of a way to take advantage of this cardinality to show non-universality. We will however be able to say more about an important class of uncountable sets, Cantor sets.

Let
$$
C = \bigcap_{n=0}^\infty C_n \subseteq [0, 1]
$$
be a \textit{symmetric} Cantor set defined as follows. We have $C_0 = [0, 1]$. The sequence $C_n$ will be a decreasing sequence and each $C_n$ is a finite union of $2^n$ disjoint closed intervals of equal length $\ell_n$. From $C_{n-1}$ we derive $C_n$ by visiting each interval of $C_{n-1}$ and removing a middle open interval\footnote{The case when one allows $d_n \le \ell_n$ leads to ``fatter'' sets which can be handled with the methods in \cite{gallagher2022topological} so we decided to exclude them from discussion in this paper to avoid unnecessary complications.} of length $d_n > \ell _n$. We write $r_n = \frac{\ell_n}{2\ell_n+d_n} = \frac{\ell_n}{\ell_{n-1}}$. It follows that $\ell_n = r_1 r_2 \cdots r_n$.

Denote by $L_n$ the set of left endpoints of the intervals in $C_n$ and by $R_n$ the set of right endpoints. We have $\Abs{L_n} = \Abs{R_n} = 2^n$ and both $L_n$ and $R_n$ are subsets of $C$. We also have $L_n \subseteq L_{n+1}$ and $R_n \subseteq R_{n+1}$.

In \cite{bourgain1987construction} it is proved that for any infinite sets $S_0, S_1, S_2 \subseteq \RR$ the set $S_0+S_1+S_2$ is not universal. Using this result we can prove that symmetric Cantor sets are not universal. Indeed we have $L_0 = \Set{0}$ and, for $n \ge 0$,
$$
L_{n+1} = L_n + \Set{0, d_{n+1}+\ell_{n+1}}.
$$
If we denote by $L = \bigcup_{n=0}^\infty L_n \subseteq C$ the set of all left interval endpoints in all stages of the construction, we have
\begin{align*}
L &= \Set{0, d_1+\ell_1} + \Set{0, d_2+\ell_2} + \Set{0, d_3+\ell_3} + \cdots\\
  &= S_0 + S_1 + S_2,
\end{align*}
where
$$
S_j = \bigoplus_{i \equiv j \bmod 3} \Set{0, d_i+\ell_i},\ \ \ j=0, 1, 2.
$$
Since the sets $S_j$ are infinite, we obtain by the mentioned result of \cite{bourgain1987construction} that $L$, and therefore $C$, is not universal\footnote{Pointed out to us by the referee.}.

Let us now show how using Theorem \ref{th:old} we can immediately prove the following, which shows how the (nearly) double exponential decay of $\ell_n$ comes into the picture with our method. This happens because the structure of a Cantor set, even if one includes only the endpoints $\bigcup_n (L_n \cup R_n)$, is much richer than that of a sequence.
\begin{theorem}\label{th:cantor-simple}
If $C$ is a symmetric Cantor set with $\ell_n < d_n$ and there is a subsequecnce $n_k \in \NN$ such that 
\beql{elln}
-\log\ell_{n_k} = o(2^{n_k})
\eeq
then $C$ is not universal.
\end{theorem}

\begin{remark}
Let us make clear here that Theorem \ref{th:cantor-simple} is suboptimal, compared to what one can prove as a corollary of Bourgain's result in \cite{bourgain1987construction}, as it does not work for all symmetric Cantor sets but only for those that are thinning out sufficiently slowly according to \eqref{elln}. The proof however is rather simple, following easily from Theorem \ref{th:old} and it demonstrates the fact that, for our method to work, one generally needs bounds on the decay. Additionally, the use of Bourgain's result is \textit{fragile}, as it depends on the algebraic properties of the set. Indeed, if one relaxes the definition of symmetric Cantor sets to allow the removed intervals to move a little bit to the right or left, differently in each interval, the additive structure evaporates and Bourgain's theorem does not apply, but it's easy to see that the proof of \ref{th:cantor-simple} goes through.
\end{remark}

\begin{proof}
The minimum distance between two points of $L_n$ is $\ell_n+d_n > \ell_n$ and $\Abs{L_n} = 2^n$. (See Fig.\ \ref{fig:one-level}.)

\begin{figure}[ht]
%\resizebox{7cm}{!}{\input cubes.pdf_t}
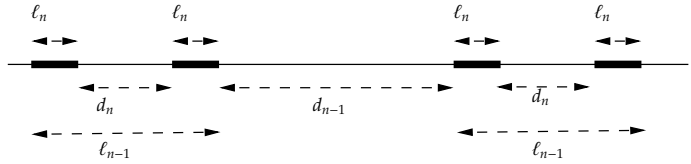
\caption{The $n$-th generation $C_n$}
\label{fig:one-level}
\end{figure}

If we have arbitrarily large $n_k$ such that \eqref{elln} holds then, applying Theorem \ref{th:old} to the set $L_{n_k}$, we obtain that $\bigcup_k L_{n_k}$, hence also its superset $C$, is not universal.
\end{proof}

%%%%%%%%%%%%%%%%%%%%%%%%%%%%%%%%%%%%%%%%%%%%%%%%%%%%%%%%%%%%%%%%%%%%%
%%%%%%%%%%%%%%%%%%%%%%%%%%%%%%%%%%%%%%%%%%%%%%%%%%%%%%%%%%%%%%%%%%%%%
\subsection{Sets of full measure}

If the set $A$ is countable and the set $E$ has full measure in $[0, 1]$ then it is easy to see that we can find an affine copy of $A$ in $E$ (even a translation copy of $t_0A$ where $\diam(t_0A)<1$). So it does not make sense, in general, to demand that the avoiding set $E$ has measure 1 in $[0, 1]$ instead of measure arbitrarily close to 1. This is not true if $A$ is uncountable. It is legitimate to try to avoid all affine copies of an uncountable set $A$ with a set $E \subseteq [0, 1]$ of measure 1.

In \cite{gallagher2022topological} the notion of universality is modified to account for topological ``size''. There, a set $A \subseteq \RR$ is called \textit{topologically universal} if one can find an affine copy of $A$ in any set $E \subseteq \RR$ which is a dense $G_\delta$ set. (They work in higher dimension as well.) By Baire's theorem all countable sets $A$ are topologically universal, so the interest shifts necessarily to uncountable sets.

They study Cantor sets (more generally than we do: a Cantor set is a totally disconnected, perfect compact set in Euclidean space) and their results are about Cantor sets of \textit{positive Newhouse thickness}. We refer to \cite[\S 2]{gallagher2022topological} for the precise definition of Newhouse thickness. For the class of symmetric Cantor sets that this paper is about the Newhouse thickness is the quantity
\beql{nt}
\inf_{n=1, 2, \ldots} \frac{\ell_n}{d_n}.
\eeq
If we take, for example, the usual ternary Cantor set, its Newhouse thickness is 1. Having positive Newhouse thickness roughly means that, at each stage, we do not throw away (from each interval) much more than we keep. Symmetric Cantor sets of positive Newhouse thickness have their $n$-th generation intervals $\ell_n$ decay no faster than exponentially.

Theorem 1.5 of \cite{gallagher2022topological} shows that there exists a dense $G_\delta$ set which is also a set of full Lebesgue measure (i.e.\ with null complement) which does not contain \textit{any} Cantor set of positive Newhouse thickness. This unexpected extreme non-universality is due to the so-called Newhouse gap lemma, which says that two Cantor sets, such that none of the two is contained in a ``gap'' of the other, always intersect if the product of their Newhouse thickness is at least 1. (See \cite[Lemma 3.6]{gallagher2022topological} and works cited therein for more.)

We cannot match in this paper the simultaneous avoidance character of the result in \cite{gallagher2022topological}, but we manage to go down into the zero-Newhouse-thickness territory. Our main result is the following. Let us stress here that the avoiding set is \textit{of full measure}, as in \cite{gallagher2022topological} and unlike all other cited work on the Erd\H os similarity problem. 
\begin{theorem}\label{th:main}
For any symmetric Cantor set $C$ with
\beql{elln-condition}
-\log\ell_n = o(2^{n^{1-\epsilon}}), \text{ for some $\epsilon>0$,}
\eeq
there exists a set $E \subseteq [0, 1]$ of Lebesgue measure 1 such that
\beql{not-contained}
(x+tC) \nsubseteq E
\eeq
for all $x, t \in \RR$, $t \neq 0$. 
\end{theorem}
We prove this result in \S\ref{sec:main}.

In Theorem \ref{th:main} we are not constructing a set of full measure that avoids all Cantor sets in any wide class, such as those of positive Newhouse thickness. But we are constructing, given a very thin Cantor set $C$, a set of full measure avoiding $C$.

\textbf{Example}. Let us take the Cantor set $C$ with $r_n = \frac{\ell_n}{\ell_{n-1}} = \frac{1}{n}$, a set with zero Newhouse thickness. Then
$$
\ell_n = r_1 \cdot r_2 \cdots r_n
$$
so
$$
-\log\ell_n = -\log 1 - \log \frac12 - \cdots - \log \frac1n = O(n \log n),
$$
so condition \eqref{elln-condition} is valid for this set and, therefore, by Theorem \ref{th:main}, there exists a set $E \subseteq [0, 1]$ of measure 1 containing no affine copy of $C$.

\begin{remark}
It is easily seen in the proof of Theorem \ref{th:main} that one does not need to impose such a rigid structure on the Cantor set $C$. For example, it is not necessary that the interval we throw away from each interval of $C_n$ is exactly in the middle of the interval. Many other relaxations of the assumptions are possible with the same method.
\end{remark}

\textbf{Acknowledgment:} I would like to thank Chun-Kit Lai for useful comments on the manuscript.

%%%%%%%%%%%%%%%%%%%%%%%%%%%%%%%%%%%%%%%%%%%%%%%%%%%%%%%%%%%%%%%%%%%%%
%%%%%%%%%%%%%%%%%%%%%%%%%%%%%%%%%%%%%%%%%%%%%%%%%%%%%%%%%%%%%%%%%%%%%
%%%%%%%%%%%%%%%%%%%%%%%%%%%%%%%%%%%%%%%%%%%%%%%%%%%%%%%%%%%%%%%%%%%%%
%%%%%%%%%%%%%%%%%%%%%%%%%%%%%%%%%%%%%%%%%%%%%%%%%%%%%%%%%%%%%%%%%%%%%

\section{Construction of a set of full measure}\label{sec:main}

In this section we prove Theorem \ref{th:main}. The proof is probabilistic, a modification of that used in Theorem 3 of \cite{kolountzakis1997infinite}.

We will construct the set $E^c = (-\infty, 0) \cup F \cup (1, +\infty)$ where $F \subseteq [0, 1]$ is a Lebesgue measurable set of measure 0, and $F$ will be such that whenever $x+tC \subseteq [0, 1]$ we have that $F \cap (x+tC) \neq \emptyset$.

Our first remark is that it is enough to construct such a set $F$ which achieves \eqref{intersect} below for all $a \le x \le b$, $A \le t \le B$, where $a, b, A, B$ are any fixed numbers (with $A, B \neq 0$ and of the same sign). Since we can exhaust the $(x, t)$ parameter space with a countable union of such $[a, b] \times [A, B]$ rectangles we can clearly take the union of all the sets $F$ corresponding to these rectangles and still have a set of measure 0.

So we assume $a, b, A, B$ are fixed from now on. We shall construct a compact null set $F \subseteq [0, 1]$ such that for all $x \in [a, b]$ and $t \in [A, B]$, for which $x+tC \subseteq [0, 1]$, we have
\beql{intersect}
(x+tC) \cap F \neq \emptyset.
\eeq

Our set $F$ will be the intersection of a decreasing sequence of compact sets $F_n \subseteq [0, 1]$, $n=0, 1, 2, \ldots$, such that $m(F_n) \to 0$, which implies $m(F) = 0$. We initially take $F_0 = [0, 1]$.

\begin{figure}[ht]
%\resizebox{7cm}{!}{\input cubes.pdf_t}
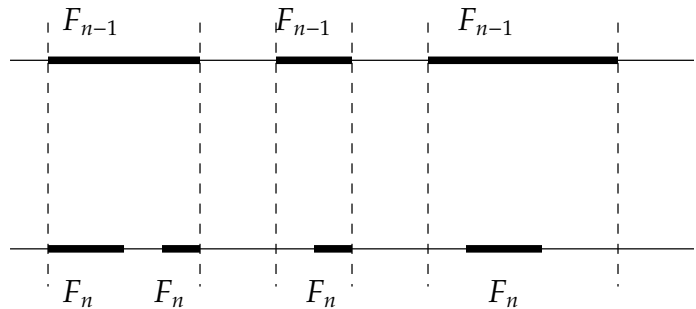
\caption{Two successive sets $F_n$. The set $F_n$ is a random subset of $F_{n-1}$.}
\label{fig:fn}
\end{figure}

In what follows we always assume for a pair of parameters $(x, t)$ that it is such that $x+tC \subseteq [0, 1]$.

To ensure $(x+tC) \cap F \neq \emptyset$ it is therefore enough, by the finite intesection property of compact sets, to ensure that $(x+tC) \cap F_n \neq \emptyset$ for all $n$. And for this it is enough to ensure that the set
$$
x+t \partial C_{\phi(n)} = x+t(L_{\phi(n)} \cup R_{\phi(n)})
$$
intersects $F_n$, where $\phi:\NN\to\NN$ is a strictly increasing function, that will be specified later.

Each set $F_n$ that we construct will be a finite union of disjoint closed intervals in $[0, 1]$. We call these the \textit{maximal} intervals of $F_n$.

To carry out our construction we will preserve the following property from $F_{n-1}$ to $F_n$.
\begin{quotation} \textbf{Property A:}
For all $(x, t) \in [a, b] \times [A, B]$ such that $x+tC \subseteq [0, 1]$ the set
$F_n$ contains both endpoints of an interval of $x+tC_{\phi(n)}$ in one of its maximal intervals.
\end{quotation}

\begin{figure}[ht]
%\resizebox{7cm}{!}{\input cubes.pdf_t}
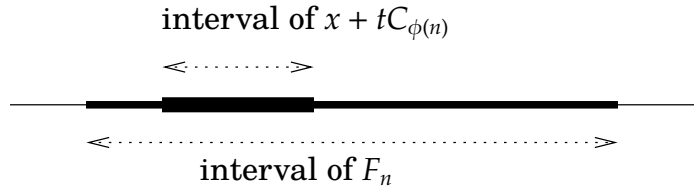
\caption{Some interval of $x+tC_{\phi(n)}$ is contained in some interval of $F_n$.}
\label{fig:interval}
\end{figure}

Each set $F_n$ will consist of a finite union of non-overlapping (but possibly sharing endpoints) closed intervals of length $f_n$ each. We derive $F_n$ from $F_{n-1}$ by subdividing $F_{n-1}$ into non-overlapping intervals of length $f_n$ and keeping (into $F_n$) each interval with probability $q_n \to 0$, (with $q_n < 1/4$ for all $n$) independently.

\begin{figure}[ht]
%\resizebox{7cm}{!}{\input cubes.pdf_t}
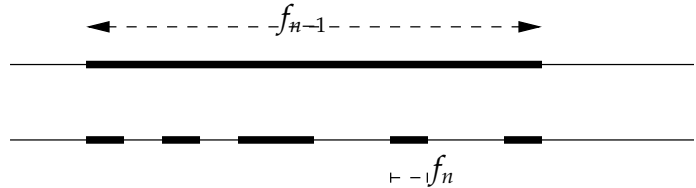
\caption{Some interval of $F_{n-1}$, of length $f_{n-1}$ giving rise to a random collection of intervals that make up $F_n$, each of them of length $f_n$.}
\label{fig:random}
\end{figure}

Assuming $F_{n-1}$ given and satisfying Property A we will prove that with positive probability the set $F_n$ will also satisfy Property A and its measure $m(F_n)$ will be at most half that of $m(F_{n-1})$.

By the randomized construction of $F_n$ from $F_{n-1}$ we deduce $\Mean{m(F_n)} = q_n m(F_{n-1})$ and by Markov's inequality we have
$$
\Prob{m(F_n) > 2 q_n m(F_{n-1})} < \frac12.
$$
So with probability at least $1/2$ we have
$$
m(F_n) \le (2 q_n) m(F_{n-1}) \le \frac12 m(F_{n-1}).
$$

Thus it suffices to show that Property A holds for $F_n$ with probability tending to 1 with $n$. This ensures that there exists a decreasing sequence of sets $F_n$ satisfying Property A.

We define $f_n = f_{n-1}/k$, where $k$ is the smallest integer so that $f_n \le 0.9 A \ell_{\phi(n)}$ (since we want $f_n$ to divide $f_{n-1}$). Notice that this implies
\beql{fn-range}
0.45 A \ell_{\phi(n)} \le f_n \le 0.9 A \ell_{\phi(n)}.
\eeq
With this choice we have made sure that each point of $x+t\partial C_{\phi(n)}$ belongs to a different $f_n$-length interval in the subdivision of $F_{n-1}$ and, therefore, that the events
$$
E_p:\ \ \ x+tp \in F_n,
$$
where $p \in \partial C_{\phi(n)}$,
are independent for each fixed choice of the parameters $x, t$.

The crucial observation here is that in order to make sure that Property A holds for all $(x, t)$ it is sufficient to check for only a finite set of pairs $(x, t)$. The parameter space $(x, t)$ is partitioned by the straight lines
\beql{lines}
x+ta=b,
\eeq
where $a \in \partial C_{\phi(n)}$ and $b$ is an endpoint of any $f_n$-length interval in $F_{n-1}$. The number $S_n$ of such straight lines is therefore
$$
S_n \le \Abs{\partial C_{\phi(n)}}\,m(F_{n-1}) f_n^{-1} \le 2^{\phi(n)} m(F_{n-1}) f_n^{-1},
$$
and these lines partition the $(x, t)$ space into $O(S_n^2)$ open, connected (since they are convex) regions. It is clear that it is enough to select one $(x,t)$ point in every such region. 

Indeed, suppose $(x_1, t_1)$ and $(x_2, t_2)$ are two points in one such region $R$ and consider the straight line segment connecting them, which lies completely in this region $R$, since $R$ is convex. Then we can move continuously from $(x_1, t_1)$ to $(x_2, t_2)$ along this straight line segment without ever leaving the region $R$. As $(x, t)$ carries out this motion the points of $x+t\partial C_{\phi(n)}$ never cross a subdivision point in $F_{n-1}$, as, if that happened, the point $(x, t)$ would be on one of the straight lines \eqref{lines}. Therefore, for each $p \in \partial C_{\phi(n)}$ the two events $x_1+t_1 p \in F_n$ and $x_2 + t_2 p \in F_n$ are either both true or both false.

\textbf{Remark about the non-interior points $(x, t)$:} The case where the point $(x, t)$ is actually on a dividing line causes no problems as we can always add to the final set $F$ that we construct the countable set of all dividing points, i.e.\ the integer multiples of all numbers $f_n$, for all $n$, thus ensuring that for any such $(x, t)$ the corresponding set $x+tC$ intersects $F$.

If Property A holds for all these finitely many points (one per region of the subdivision) then it holds for all pairs $(x, t)$. Thus the number of $(x,t)$ points that we have to check is
$$
O(S_n^2) = O(2^{2 \phi(n)}) m(F_{n-1})^2 f_n^{-2}  = O(2^{2 \phi(n)} \ell_{\phi(n)}^{-2}).
$$
For each of these points $(x, t)$ the probability that $F_n$ does not contain both endpoints of some interval of $x+tC_{\phi(n)}$ in one of its maximal intervals is, because of independence,
$$
\le (1-q_n^K)^{p_n(x, t)},
$$
where $p_n(x, t)$ is the number of intervals of $x+tC_{\phi(n)}$ which are contained in some maximal interval of $F_{n-1}$ and the positive integer $K$ is the maximum number of $f_n$-length intervals that are required to cover an interval of $x+tC_{\phi(n)}$. Since these intervals have maximum length (as $t$ varies in $[A, B]$) equal to $B\ell_{\phi(n)}$ and since $f_n \ge 0.45 A \ell_{\phi(n)}$, by \eqref{fn-range}, it follows that $K \le 3 B/A$.

We now observe that $p_n(x, t) \ge 2^{\phi(n)-\phi(n-1)}$, as there are $\phi(n)-\phi(n-1)$ generations of the Cantor set between $x+tC_{\phi(n-1)}$ and $x+tC_{\phi(n)}$ and each generation doubles the number of intervals and there is at least one interval of $x+tC_{\phi(n-1)}$ contained in some maximal interval of $F_{n-1}$ (by Property A, which we assume true for $n-1$). 

\begin{figure}[ht]
%\resizebox{7cm}{!}{\input cubes.pdf_t}
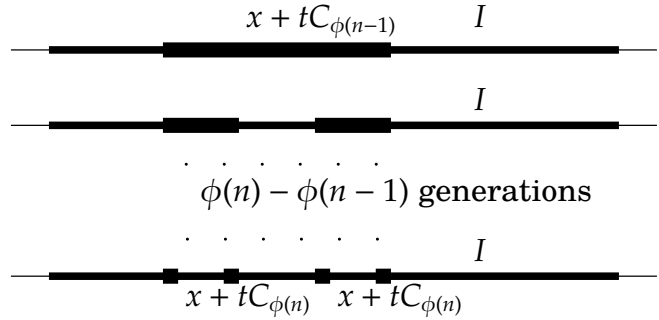
\caption{If some interval of $x+tC_{\phi(n-1)}$ is contained in some interval $I$ of $F_{n-1}$ then at least $2^{\phi(n)-\phi(n-1)}$ intervals of $x+tC_{\phi(n)}$ are contained in $I$.}
\label{fig:tree}
\end{figure}

The total bad probability then, that is the probability that $F_n$ (constructed at random from $F_{n-1}$ which is considered fixed and satisfying Property A) will fail to satisfy Property A, is at most
$$
(1-q_n^K)^{2^{\phi(n)-\phi(n-1)}}2^{2 \phi(n)} \ell_{\phi(n)}^{-2}.
$$
We would like this quantity to go to 0 with $n$. Taking logarithms we have that the logarithm of the above is
$$
\log(1-q_n^K) \cdot 2^{\phi(n)-\phi(n-1)} + O(\phi(n)) + O(-\log\ell_{\phi(n)}).
$$
Define now $\phi(n) = \Floor{n^{1+\eta}}$ for some positive $\eta$. This gives $\phi(n)-\phi(n-1) \ge C n^\eta$. It is enough therefore to have
$$
\log(1-q_n^K) 2^{C n^\eta} + O(n^{1+\eta}) + O(-\log\ell_{\Floor{n^{1+\eta}}}) \to -\infty.
$$
Only the first term is negative. As we can take $q_n \to 0$ as slowly as we please it follows that the above quantity tends to $-\infty$ if we have
$$
-\log\ell_{\Floor{n^{1+\eta}}} = o(2^{C n^\eta}),
$$
which is true if pick $\eta$ so that $\epsilon=\frac{1}{1+\eta}$ using \eqref{elln-condition}.

This concludes the proof of Theorem \ref{th:main}.

\bibliographystyle{amsalpha}
\bibliography{universal-sets}

\newcommand{\etalchar}[1]{$^{#1}$}
\providecommand{\bysame}{\leavevmode\hbox to3em{\hrulefill}\thinspace}
\providecommand{\MR}{\relax\ifhmode\unskip\space\fi MR }
\providecommand{\MRhref}[2]{%
  \href{http://www.ams.org/mathscinet-getitem?mr=#1}{#2}
}
\providecommand{\href}[2]{#2}
\begin{thebibliography}{BGK{\etalchar{+}}22}

\bibitem[BGK{\etalchar{+}}22]{burgin2022large}
Alex Burgin, Samuel Goldberg, Tam{\'a}s Keleti, Connor MacMahon, and Xianzhi
  Wang, \emph{Large sets avoiding infinite arithmetic/geometric progressions},
  arXiv preprint arXiv:2210.09284 (2022).

\bibitem[BKM22]{bradford2022large}
Laurestine Bradford, Hannah Kohut, and Yuveshen Mooroogen, \emph{{Large Subsets
  of Euclidean Space Avoiding Infinite Arithmetic Progressions}}, arXiv
  preprint arXiv:2205.04786 (2022).

\bibitem[Bou87]{bourgain1987construction}
Jean Bourgain, \emph{Construction of sets of positive measure not containing an
  affine image of a given infinite structure}, Israel Journal of Mathematics
  \textbf{60} (1987), no.~3, 333--344.

\bibitem[Chl15]{chlebik2015erdos}
Miroslav Chlebik, \emph{{On the Erd\H os similarity problem}}, arXiv preprint
  arXiv:1512.05607 (2015).

\bibitem[CLP22]{cruz2022large}
Angel Cruz, Chun-Kit Lai, and Malabika Pramanik, \emph{Large sets avoiding
  affine copies of infinite sequences}, arXiv preprint arXiv:2204.12720 (2022).

\bibitem[DPZ21]{denson2021large}
Jacob Denson, Malabika Pramanik, and Joshua Zahl, \emph{Large sets avoiding
  rough patterns}, Harmonic analysis and applications, Springer, 2021,
  pp.~59--75.

\bibitem[Eig85]{eigen1985putting}
SJ~Eigen, \emph{Putting convergent sequences into measurable sets}, Studia Sci.
  Math. Hung \textbf{20} (1985), 411--412.

\bibitem[Fal84]{falconer1984problem}
KJ~Falconer, \emph{{On a problem of Erd{\H{o}}s on sequences and measurable
  sets}}, Proceedings of the American Mathematical Society \textbf{90} (1984),
  no.~1, 77--78.

\bibitem[FP18]{fraser2018large}
Robert Fraser and Malabika Pramanik, \emph{Large sets avoiding patterns},
  Analysis \& PDE \textbf{11} (2018), no.~5, 1083--1111.

\bibitem[GLW22]{gallagher2022topological}
John Gallagher, Chun-Kit Lai, and Eric Weber, \emph{{On a topological Erd\H os
  similarity problem}}, arXiv preprint arXiv:2207.03077 (2022).

\bibitem[HL98]{humke1998visit}
Paul Humke and Mikl{\'o}s Laczkovich, \emph{{A visit to the Erd\H os problem}},
  Proceedings of the American Mathematical Society \textbf{126} (1998), no.~3,
  819--822.

\bibitem[Kol97]{kolountzakis1997infinite}
Mihail~N Kolountzakis, \emph{Infinite patterns that can be avoided by measure},
  Bulletin of the London Mathematical Society \textbf{29} (1997), no.~4,
  415--424.

\bibitem[Kom83]{komjath1983large}
P{\'e}ter Komj{\'a}th, \emph{Large sets not containing images of a given
  sequence}, Canadian Mathematical Bulletin \textbf{26} (1983), no.~1, 41--43.

\bibitem[KP22]{kolountzakis2022large}
Mihail~N. Kolountzakis and Effie Papageorgiou, \emph{Large sets containing no
  copies of a given infinite sequence}, arXiv preprint arXiv:2208.02637 (2022).

\bibitem[Mag11]{maga2011full}
P{\'e}ter Maga, \emph{Full dimensional sets without given patterns}, Real
  Analysis Exchange \textbf{36} (2011), no.~1, 79--90.

\bibitem[M{\'a}t17]{mathe2017sets}
Andr{\'a}s M{\'a}th{\'e}, \emph{Sets of large dimension not containing
  polynomial configurations}, Advances in Mathematics \textbf{316} (2017),
  691--709.

\bibitem[Shm17]{shmerkin2017salem}
Pablo Shmerkin, \emph{Salem sets with no arithmetic progressions},
  International Mathematics Research Notices \textbf{2017} (2017), no.~7,
  1929--1941.

\bibitem[Sve00]{svetic2000erdHos}
RE~Svetic, \emph{The erd{\H{o}}s similarity problem: a survey}, Real Analysis
  Exchange (2000), 525--539.

\bibitem[Yav21]{yavicoli2021large}
Alexia Yavicoli, \emph{Large sets avoiding linear patterns}, Proceedings of the
  American Mathematical Society \textbf{149} (2021), no.~10, 4057--4066.

\end{thebibliography}

\end{document}